\documentclass[12pt]{amsart}
\usepackage{amssymb,amsmath}
\usepackage{pgfplots}
\usepackage{array}
\pgfplotsset{compat=newest}
\usepackage[english]{babel}
\usepackage{hyperref}
\hypersetup{colorlinks=true,
 linkcolor=blue,
filecolor=magenta,urlcolor=blue}

\makeatletter
\theoremstyle{plain}
 \newtheorem{thm}{Theorem}[section]
\newtheorem{thm*}{Theorem}
 \newtheorem{lem}[thm]{Lemma}
 \newtheorem{prop}[thm]{Proposition}
 \newtheorem{cor}[thm]{Corollary}
 \numberwithin{equation}{section} %% Comment out for sequentially-numbered
\numberwithin{figure}{section} %% Comment out for sequentially-numbered
 \theoremstyle{plain}
 \theoremstyle{definition}
 \newtheorem{defn}[thm]{Definition}

\newcommand{\cC}{{{\mathcal C}}}

\newcommand{\fH}{{{\mathfrak H}}}

\newcommand{\cP}{{{\mathcal P}}}

\newcommand{\C}{{{\mathbb C}}}

\newcommand{\R}{{{\mathbb R}}}

\makeatother

\date{\today
}
\begin{document}

\title[Equilateral dimension of $\fH$]{Equilateral dimension of the Heisenberg group}

\author{J.~Kim \& I.D. Platis}

\begin{abstract}
Let $\fH$ be the first Heisenberg group equipped with the Kor\'anyi metric $d$. We prove that the equilateral dimension of $\fH$ is 4. 
\end{abstract}

\address{Department of Mathematics Education, Chungnam National University, Daejeon, 34134, South Korea}
\email{calvary@cnu.ac.kr}
\address{Department of Mathematics and Applied Mathematics, University of Crete, Heraklion, GR-70013, Greece}
\email{jplatis@math.uoc.gr}

\maketitle

\tableofcontents

\section{Introduction}
Let $(X,d)$ be a metric space. %There are various ways to associate a {\it dimension} to $X$-the Hausdorff dimension is perhaps one of the most well-known.
 The {\it equilateral dimension} $\dim_{\rm{E}}(X)$ of a $X$ is the maximum number of points that are all at equal distances from each other.
The equilateral dimension of the $n$-dimensional Euclidean space $(\R^n,e)$, where $e$ is the standard Euclidean metric,  is known to be $n + 1$; it is achieved by a regular simplex.  In the case of the $n$-dimensional vector space equipped with the $L^{\infty}$ norm the equilateral dimension  is $ 2^{n}$ and it is achieved by a hypercube, see \cite{DG} as well as \cite{P} and \cite{So}. 
However, the equilateral dimension of an $n$-dimensional vector space equipped with the $ L^1$ norm is not known; Kusner's conjecture states (see \cite{S2}) that it is exactly $2^n$, achieved by a cross-polytope.

There has been an extensive study on the  equilateral dimension of $L^p$-spaces, $1<p<\infty$; in particular, for $p=2$, $\dim_{\rm{E}}(X)=n+1$ and for  $2< p<\infty$, $\dim_{\rm{E}}(X)\ge n+1$ (see for instance  \cite{SV}). Various results also exist  for arbitrary normed spaces, illustrative references are \cite{B}, \cite{S}. Also,
for $n$-dimensional Riemannian manifolds $M$, we have $\dim_{\rm{E}}(M)=n+1$ (for manifolds with bounded curvature see for instance \cite{M}), and for Minkowski spaces we refer to \cite{P}. 

Generally speaking, to track down the equilateral dimension of an arbitrary metric space is not an obvious task; to have a clear understanding of the properties of its similarity group is certainly helpful. As an elementary example we will present a short proof of the fact that $\dim_{\rm{E}}(\R^3,e)=4$. Since dilations are within the similarity group of $(\R^3,e)$, we may normalise any equilateral set $S=\{p_0,p_1,\dots\}$ so that $e(p_i,p_j)=1$. We may further normalise so that $p_0={\bf 0}=(0,0,0)$ and $p_1={\bf 1}=(1,0,0)$ so that all points of $S$ lie on the Euclidean unit sphere $S^2$. If now $p\in S\setminus\{{\bf 0},{\bf 1}\}$, $p=(x,y,z)$, then the conditions $e(p,{\bf 0})=e(p,{\bf 1})=1$ lead us to
$$
x=1/2,\quad y^2+z^2=3/4,
$$  
that is, all points of $S$ lie on the above circle which is the intersection of the unit sphere and the sphere centred at ${\bf 1}$ and of radius 1. This circle is centred in $(1/2,0,0)$ and its radius is $\sqrt{3}/2$. From an elementary argument we then find that there can be three points on a circle which are at distance 1 from one another only if and only if its radius is $r=\sqrt{3}/3$. %Suppose now that two points $p$ and $q=(1/2,v,w)$ lying in that circle and satisfying $e(p,q)=1$. Denote by $a$ the angle of vectors $p$,$q$. By combining all previous relations we obtain
%$
%\cos a= 1/3.
%$ 
Therefore there can be at most two points on this circle having distance equal to 1, hence $\dim_{\rm{E}}(\R^3)=4$.

We will follow the rationale of this elementary approach in the case of the Heisenberg group $\fH$ which is the object of our study in this article (for the definition, see Section \ref{sec-heis}).
Contrary to the Euclidean case and to all cases mentioned before, $\fH$ is equipped with a metric which is not induced by a norm. Thus we are not able to use general arguments from the study of equilateral dimension of finite normed spaces; rather, we concentrate in the use of the properties of the similarity group of $\fH$. In this article we prove: 
\begin{thm}\label{thm-main}
The equilateral dimension $\dim_{\rm{E}}(\fH)$ of $\fH$ is 4.
\end{thm} 
Therefore   $\dim_{\rm{E}}(\fH)$ is equal to its Hausdorff dimension. This theorem answers to the negative a question posed in \cite{CT} on whether there exist five equilateral points in $\fH$. We also mention that in that article there exist examples of equilateral subsets comprising of three points. 

This paper is organised as follows. In Section \ref{sec:setup} we state well known facts about the Heisenberg group and the Kor\'anyi metric. The equilateral dimension of $\C$-circles is also calculated explicitly in \ref{prop-c-equid}. Section \ref{sec:main} is the section where Theorem \ref{thm-main} is proved. For the proof, several steps are taken; first, we study the case of equilateral sets where three of their points lie in the same $\C$-circle (Theorem \ref{thm-4max3}). Then, we proceed to the case of equilateral sets such that only two of their points lie in the same $\C$-circle. Since a $\C$-circle can be either infinite or finite, we distinguish cases again: We study the infinite $\C$-circle case in Section \ref{sec-2p-inf} (Theorem \ref{prop:fin}) and the finite $\C$-circle case in Section \ref{sec-p-fin} (Theorem \ref{prop:fin2}). The proof of Theorem \ref{thm-main} follows from combining all the aforementioned theorems. Although elementary, the proof of Theorem \ref{prop:fin2} in particular, involves a series of long but mostly straightforward calculations. Some of them are contained into the Appendix.

{\it Acknowledgements.} We wish to thank Vassilis Chousionis for proposing the problem to the second author and J.R. Parker for useful observations. Part of this work was carried out while the first author was visiting University of Crete and the second author was visiting University of Ioannina. Hospitality in both cases is gratefully appreciated. The first author was supported by 2021 Chungnam National University fund. %Also, for the images in this article we used Desmos (\texttt{https://www.desmos.com/}).

\section{Preliminaries}\label{sec:setup}
In Section \ref{sec-heis} we define the metric space $(\fH,d)$, where $\fH$ is the Heisenberg group and $d$ is the  Kor\'anyi metric, and we describe the similarity group $G$ of  $(\fH,d)$. We refer the reader to the standard book \cite{Gol} for an extensive study of $\fH$. In Section \ref{sec-korsphere} we give a distance formula for points on the unit Kor\'anyi sphere. Finally, in Section \ref{sec-c-circles} we calculate the equilateral dimension of $\C$-circles (Proposition \ref{prop-c-equid}). 
\subsection{Heisenberg group}\label{sec-heis}
By $\fH$ we shall denote the Heisenberg group; recall that $\fH$ is $\C\times\R$ with group law:
$$
(z,t)\star (w,s)=(z+w,t+s+2\Im(z\overline{w})).
$$
The Heisenberg norm (Kor\'anyi gauge) is given by
$$
\left|(z,t)\right|_\fH=| |z|^2-it|^{1/2},
$$
Despite its name, $|\cdot|_\fH$ is not a norm in the usual sense. But from this we obtain a distance $d$ is given by
\begin{equation}\label{eq-d}
d\left((z,t),\,(w,s)\right)
=\left|(w,s)^{-1}\star (z,t)\right|_\fH=\left(|z-w|^4+(t-s+2\Im(z\overline{w}))^2\right)^{1/4}.
\end{equation}
This is the {\it Kor\'anyi distance} and it is invariant under: 
\begin{enumerate}
\item [{a)}] left-translations $L_{w,s)}$ of $\fH$, $$(z,t)\to(w,s)\star (z,t);$$
\item [{b)}] rotations $R_\phi$ around the $t$-axis  $$(z,t)\mapsto(ze^{i\phi},t)\quad\phi\in\R;$$
\item [{c)}] conjugation $j$, $$j(z,t)=(\overline{z},-t).$$.
\end{enumerate}
The distance $d$ is also scaled up to multiplicative constants by the action of Heisenberg dilations $(z,t)\mapsto$ $(rz,r^2t)$, $r\in\R_*$. The {\it similarity group}
${G=\rm Sim}(\fH, d)\simeq\fH\times\R\times \R_{>0},$ comprises all the above transformations. We note that $d$ is not a path metric. Moreover, there is no Riemannian structure in $\fH$ with underlying metric $d$.  The {\it Carnot-Carath\'eodory metric} in $\fH$ which corresponds to its sub-Riemannian structure will not be of our concern here.
\subsubsection{The unit Kor\'anyi sphere and a distance formula}\label{sec-korsphere}
A general Kor\'anyi sphere $S_\fH^r(p_0)$ centred at $p_0$ with radius $r$ is the locus comprising of points $p\in\fH$ such that $d(p,p_0)=r$. The {\it Kor\'anyi unit sphere} $$
S^1_\fH=\{p\in\fH\;|\;d(p,o)=1\},$$ where $o=(0,0)$, is the surface described by the equation
$$
|z|^4+t^2=1.
$$
When $p_0=(z_0,t_0)$ and $p=(z,t)$ are points in $S^1_\fH$, then (\ref{eq-d}) takes the form
\begin{eqnarray}\label{eq-d-s}
d^4(p,p_0)&=&2+6|z|^2|z_0|^2-2tt_0\\
\notag &&-4(|z|^2+|z_0|^2)\Re(z\overline{z_0})+4(t-t_0)\Im(z\overline{z_0}).
\end{eqnarray}
The above equation may be simplified even further by using {\it Kor\'anyi-Reimann coordinates} (see \cite{KR}):
$$
z=\sqrt{\cos\theta}e^{i\phi},\quad t=\sin\theta,
$$
where $(\theta,\phi)\in[-\pi/2,\pi/2]\times[-\pi,\pi)$. In fact, Equation (\ref{eq-d-s}) then reads as
\begin{eqnarray}\label{eq-dist-sph}
\notag d^4(p,p_0)&=&2+6\cos\theta\cos\theta_0-2\sin\theta\sin\theta_0\\
\label{eq-d-s-K} &&-8\sqrt{\cos\theta\cos\theta_0}\cos((\theta+\theta_0)/2)\cos\left((\phi+\theta/2)-(\phi_0+\theta_0/2)\right).
\end{eqnarray}
For fixed $p_0\in S^1_\fH$ and $p\in S^1_\fH$, the locus $d(p,p_0)=1$ is the intersection of $S^1_\fH$ with the Kor\'anyi sphere centred at $p_0$ with radius 1. Using (\ref{eq-dist-sph}) we find that this locus is a curve on $S^1_\fH$ given parametrically by
\begin{equation}\label{eq-sph-curve}
\cos\left((\phi+\theta/2)-(\phi_0+\theta_0/2)\right)=\frac{1+6\cos\theta\cos\theta_0-2\sin\theta\sin\theta_0}{8\sqrt{\cos\theta\cos\theta_0}\cos((\theta+\theta_0)/2)},
\end{equation}
if $p_0\neq(0,\pm 1)$. If $p_0=(0,\pm 1)$, this curve is the Euclidean circle $\theta=\pm\pi/6$, respectively.
\subsubsection{Convention}  Let $S\subset\fH$ be a $d_0$ equilateral set of points: $d(p,q)=d_0$ for each $p,q\in S$. By applying the dilation $D_{1/d_0}$ we may always suppose that $S$ is a 1-equilateral set, or simply {\it equilateral set} from now on. Throughout this paper we will further suppose that we deal with such sets.

\subsection{$\C$-circles and their equilateral dimension}\label{sec-c-circles}
Recall that a $\C$-circle in $\fH$ is a topological circle which is the boundary of a complex geodesic of the underlying 2-dimensional complex hyperbolic Siegel domain. Such a circle comes in two flavours: a) {\it infinite $\C$-circles} when one of the endpoints of the complex geodesic is $\infty$ and b) {\it finite $\C$-circles} when none of the endpoints of the complex geodesic is $\infty$. Infinite $\C$-circles are lines vertical to the plane $t=0$; they are all $G$-images of the $\C$-circle 
\begin{equation}\label{eq-c-inf}
\cC_\infty=\{p=(z,t)\in\fH\;|\;z=0\}.
\end{equation}
 On the other hand, finite $\C$-circles are ellipses which are all $G$ images of the Euclidean circle 
\begin{equation}\label{eq-c-fin}
 \cC_1=\{p=(z,t)\in\fH\;|\;|z|=1,\;t=0\}.
 \end{equation}
We note here that a) there is no element of $G$ that can map an infinite $\C$-circle to a finite $\C$-circle or vice versa and b) any two points of $\fH$ lie in a $\C$-circle. For details on $C$-circles we refer to \cite{Gol}.
\begin{lem}\label{lem-ayxcircle}
There can be at most two equilateral points lying in the same infinite $\C$-circle and at most three equilateral points lying in the same finite $\C$-circle.
\end{lem}
\begin{proof}
For the first statement, we may normalize so that the infinite $\C$-circle is the vertical axis which comprises points $(0,t)$. If $S=\{p_1,p_2,\dots\}$ is an equilateral set of points on that $\C$-circle, by applying a suitable dilation and a vertical translation if necessary we may assume that $p_1=(0,-1/2)$ and $p_2=(0,1/2)$ so that $d(p_1,p_2)=1$. Now, if $(0,t)$ is equidistant from $p_1$ and $p_2$ , short calculations lead $t=0$. But this cannot be the case since $d((0,0),\;p_1)=d((0,0),\;p_2)=(1/4)^{1/4}\neq 1$.

For the second statement, we may normalize so that the finite $\C$-circle is the planar circle $|z|=r$. If $S=\{p_1,p_2,\dots\}$ is an equilateral set of points on that $\C$-circle, by performing a rotation we may suppose that $p_1=(r,0)$. Points $p=(re^{i\theta},0)$ in the circle such that $d(p,p_1)=1$ must then satisfy the relation
$$
8r^4(1-\cos\theta)=16r^4\sin^2(\theta/2)=1.
$$
This is meaningful only when $r\ge 1/2$. In that case we obtain points $p=(re^{i\theta},0)$, with
$$
\theta=\pm\arccos\left(1-\frac{1}{8r^4}\right)=\pm\theta_0.
$$
Now it is straightforward to show that $d\left((re^{i\theta_0},0),\;(re^{-i\theta_0},0)\right)=1$ if and only if $r=12^{-1/4}$ (and $\theta_0=2\pi/3$). In all other cases $d\left((re^{i\theta_0},0),\;(re^{-i\theta_0},0)\right)>1$ which proves that there can be no equilateral triple of points in this case.  
\end{proof}
From Lemma \ref{lem-ayxcircle} we immediately have
\begin{prop}\label{prop-c-equid}
Let $C$ be a $\C$-circle in $\fH$. Then:
\begin{enumerate}
\item $dim_{\rm{E}}(C)=2$ if $C$ is infinite;
\item $dim_{\rm{E}}(C)=3$ if $C$ is the $G$-image of a Euclidean circle centred at the origin and with radius $r=12^{-1/4}$. If $C$ is any other finite $\C$-circle then $dim_{\rm{E}}(C)=2$.
\end{enumerate} 
\end{prop}

\section{Equilateral dimension of $\fH$}\label{sec:main}
In Section \ref{sec-4max3} we study equilateral sets of points such that three of them lie in the same $\C$-circle. By Lemma \ref{lem-ayxcircle}, this $\C$-circle has to be finite. We prove in Theorem \ref{thm-4max3} that there can be at most four equidistant points such that three of them lie in the same $\C$-circle. In Section \ref{sec-2p-inf} we prove in Theorem \ref{prop:fin} that
the maximum number of equilateral points in $\fH$ such that two of them lie in an infinite $\C$-circle is 4.  Finally in Section \ref{sec-p-fin} we prove our result for the case where two points lie in a finite $\C$-circle in Theorem \ref{prop:fin2}.
\subsection{Three points in the same finite $\C$-circle}\label{sec-4max3}
Given an equilateral set $S=\{p_1,p_2,p_3\}$ of  points lying in the same finite $\C$-circle $\cC$, let $\cP$ denote the plane such that $S\subset \cP$. By applying a Heisenberg translation we may suppose that $\cP=\{(z,t)\in\fH\;|\;t=0\}$ and that $\cC$ is a Euclidean circle centred at the origin. By Proposition \ref{prop-c-equid}  the radius of this circle is necessarily $r_0=12^{-1/4}$ and by applying a suitable dilation and a rotation if necessary, we may also normalise so that
$$
p_1=(r_0,0),\quad p_2=r_0e^{i\theta_0},\quad p_3=r_0e^{-i\theta_0},
$$
where $\theta_0=2\pi/3$.
\begin{defn}
The canonical 3-equilateral set is the set $$S_3^{can}=\{(r_0,0),(r_0e^{i\theta_0},0),(r_0e^{-i\theta_0},0)\}.$$
\end{defn}
\begin{lem}\label{lem-4can}
The only points in $\fH$ equidistant from $S_3^{can}$ are the points $(0\pm t_0)$, where $t_0=\sqrt{11/12}$.
\end{lem}
\begin{proof}
The proof is by elementary calculations. If $(z,t)\in\fH$ such that $d((z,t),(r_0,0))=1$, then
\begin{equation}\label{eq-4can1}
|z-r_0|^4+(t+2r_0 \Im(z))^2=1.
\end{equation}
Also, $d((z,t),(r_0e^{\pm i\theta_0},0))=1$ gives
\begin{equation}\label{eq-4can2}
|z-r_0e^{\pm i\theta_0}|^4+(t+2r_0 \Im(ze^{\mp i\theta_0}))^2=1.
\end{equation}
After expanding, we subtract (\ref{eq-4can2}) from (\ref{eq-4can1}) to obtain
$$
\Re\left((|z|^2+12^{-1/2}+it)]\cdot z(3\pm i\sqrt{3})\right)=0.
$$
In other words, we have the system of equations
\begin{eqnarray*}
&&
(3x-\sqrt{3}y)(|z|^2+12^{-1/2})-(3y+\sqrt{3}x)t=0,\\
&&
(3x+\sqrt{3}y)(|z|^2+12^{-1/2})-(3y-\sqrt{3}x)t=0.
\end{eqnarray*}
Viewing this system as a homogeneous linear system in variables $|z|^2+12^{-1/2}$ and $t$, we find that if $z\neq 0$ then we must have $t=|z|^2+12^{-1/2}=0$, a contradiction.
Thus $z=0$ and we obtain from both Eqs. (\ref{eq-4can1}) and (\ref{eq-4can2}) that $t=\pm t_0=\sqrt{11/12}$.
\end{proof}
\begin{defn}
The canonical 4-equilateral set $S_4^{can}$ is the set $S_3^{can}\cup\{(0,t_0)\}$.
\end{defn}
Any set of four equidistant points in $\fH$ such that three of them lie in the same finite $\C$-circle is $G$-equivalent to the canonical set $S_4^{can}$.
If $S=\{p_1,p_2,p_3,p_4\}$ and $p_i$, $i=1,2,3$ lie in the same finite $\C$-circle, then  the subset $S_0=\{p_1,p_2,p_3\}$  is $G$-equivalent to the set $S_3^{can}$. From Lemma \ref{lem-4can} we then have that the point $p_4$ is mapped either to $(0,t_0)$ or $(0,-t_0)$. By applying a conjugation if necessary, we may normalise so that the wished point is $(0,t_0)$.

\begin{thm}\label{thm-4max3}
The maximal number of equidistant points in $\fH$ such that three of them lie in the same finite $\C$-circle is 4. 
\end{thm}

\begin{proof}
Suppose on the contrary that there exists an equilateral set $S=\{p_1,p_2,p_3,p_4,p_5\}$ such that  $p_i$, $i=1,2,3$, lie in the same finite $\C$-circle. Then the subset $S_0=\{p_1,p_2,p_3,p_4\}$ is $G$-equivalent to the set $S_4^{can}$. Since there exist no point equidistant to $S_4^{can}$, we obtain a contradiction ($d((0,t_0),\;(0,-t_0))\neq 1$).   
\end{proof}

\subsection{Two points in the same infinite $\C$-circle}\label{sec-2p-inf}
In this section we shall prove
\begin{thm}\label{prop:fin}
The maximum number of equidistant points in $\fH$ such that two of them lie in an infinite $\C$-circle is 4.
\end{thm}
\begin{proof}
Let $S=\{p_1,p_2,\dots\}$ be an equilateral set such that $p_1$ and $p_2$ lie in the same infinite $\C$-circle. By applying a left translation, a vertical translation if necessary, we may suppose that $p_1=(0,-1/2)$ and $p_2=(0,1/2)$. All points equidistant to $p_1$ and $p_2$ then lie in the $\C$-circle 
$$
\cC=\{(z,0)\in \fH\;|\;|z|=(3/4)^{1/4}\}.
$$
Since the radius of $\cC$ is greater than $12^{-1/4}$, our result follows from Proposition \ref{prop-c-equid}.
\end{proof}

\subsection{Two points in the same finite $\C$-circle}\label{sec-p-fin}
In this section we shall prove
\begin{thm}\label{prop:fin2}
The maximum number of equidistant points in $\fH$ such that at most two of them lie in the same finite $\C$-circle is 4.
\end{thm}
We consider an equilateral set $S=\{p_1,p_2,\dots...\}$ such that there are no three points lying in the same finite $\C$-circle. By letting an element of $G$ acting on $S$ we may normalize so that 
$$
p_1=(0,0), \quad p_2=(1,0).
$$
Since all other points in $S$ must lie in the unit sphere $S^1_\fH=\{(z,t)\in\fH\;|\; |z|^4+t^2=1\}$,  we shall use Kor\'anyi-Reimann spherical coordinates
$$
z=\sqrt{\cos\theta}\;e^{i\phi}, \quad t=\sin\theta,\;\; (\theta,\phi)\in\Pi=(-\pi/2,\pi/2)^2,
$$
to describe the points of $S$ from now on. Together with the notation $p=(z,t)$ we shall also use the notation $p=(\theta,\phi)$. Hence if $p=(\theta,\phi)\in S$ other than $p_1,p_2$, then by Equation \ref{eq-sph-curve} the condition $d(p,p_2)=1$ reads as
\begin{equation}\label{eq-c3fin}
1+6\cos\theta-8\sqrt{\cos\theta}\cos(\theta/2)\cos(\phi+\theta/2)=0.
\end{equation}
Equation \ref{eq-c3fin} represents a spherical curve which we shall denote by $C_3^{fin}$. Points in $C_3^{fin}$ are all of distance 1 from $p_1$ and $p_2$, that is, $C^{fin}_3$ is the intersection of $S^1_\fH$ and $S^1_\fH(1,0)$. In the appendix (Section \ref{sec-parametric}) we describe explicitly a parametric representation of $C^{fin}_3$: The curve $C_3^{fin}$ is the union of the graphs of the functions $\phi^\pm:I\to\R$ where
\begin{equation}\label{eq:phipm}
\phi^\pm(\theta)=-\frac{\theta}{2}\pm\arccos(f(\theta)) %\pm\arccos\left(\frac{1+6\cos\theta}{8\sqrt{\cos\theta}\cos(\theta/2)}\right)
\end{equation} 
and 
\begin{equation}\label{eq:f}
f(\theta)=\frac{1+6\cos\theta}{8\sqrt{\cos\theta}\cos(\theta/2)},\quad \theta\in I.
\end{equation}
Here, $I=[-\theta^{*}_0,\theta_0^{*}]$, $\theta^{*}_0=\arccos(5/2-\sqrt{6})$.
\begin{center}
\includegraphics[scale=0.3]{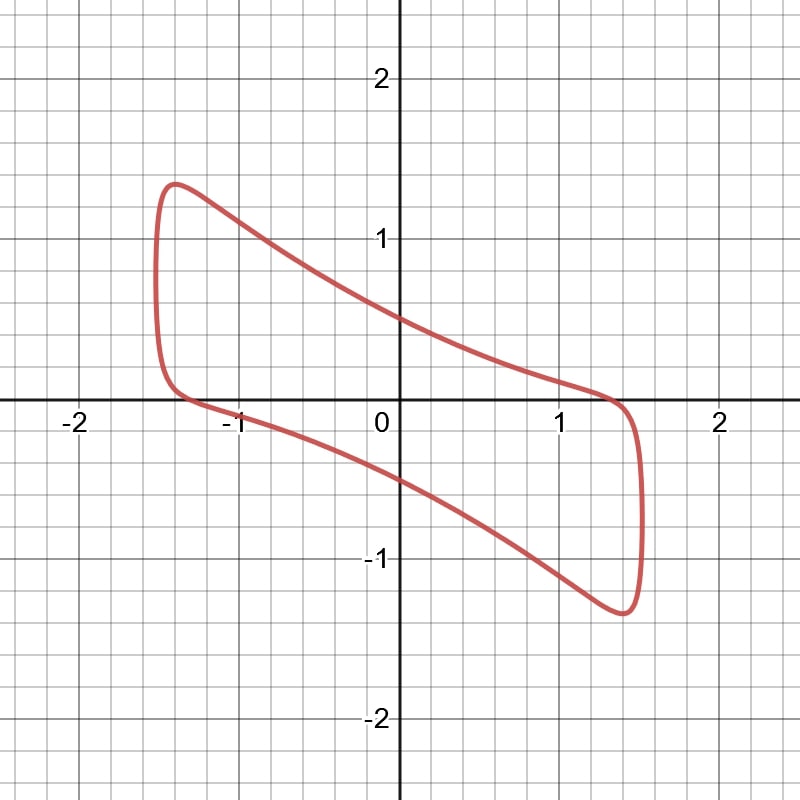}
\begin{figure}[!h]
\centering{Parametric representation of $C^{fin}_3$ in $(-\pi/2,\pi/2)^2.$ 
}
\end{figure}
\end{center}
We wish to note at this point that there are some rather obvious symmetries of $C_3^{fin}$. Namely, we have the following three involutions:
\begin{itemize}
\item The {\it antipodal involution} $j_1$:
$$
p=(\theta,\phi)\mapsto j_1(p)=(-\theta,-\phi).
$$
\item the {\it symmetric} involution $j_2$: 
$$
p=(\theta,\phi)\mapsto j_2(p)=(-\theta,\phi+\theta).$$
\item The {\it vertical} involution $j_3$: 
$$
p=(\theta,\phi)\mapsto j_3(p)=(\theta,-\phi-\theta).$$
\end{itemize}
These involutions satisfy the relation
$$
j_1\circ j_2=j_2\circ j_1=j_3
$$
and they are studied in detail in the appendix (Section \ref{sec-appendix}). We mention here that we have the following result concerning antipodal points: For all $p\in C^{fin}_3$, 
\begin{equation}\label{eq-antig1}
1.378\approx(2\sqrt{6}-3)^{1/2}\le d(p,j_1(p))\le (15/4)^{1/4}\approx 1.391.
\end{equation}
In both cases of symmetric and vertical points there exist pairs of points with distance 1. From this we establish the existence of equilateral quadruples of points which are such that two of them lie in the same finite $\C$-circle.
The next lemma shows that something much stronger holds.
 \begin{lem}\label{lem-auxfin2}
For each fixed point $p_0\in C_3^{fin}$ there exist at least two points $p\in C_3^{fin}$ such that $d(p_0,p)=1$. 
\end{lem}
\begin{proof}
Fix a point $p_0\in C_3^{fin}$ and consider the function $G_{p_0}:C_3^{fin}\to \R_{+}$ given by 
$$
G_{p_0}(p)=d(p,p_0).
$$
This function is continuous and bounded on both branches of $C_3^{fin}$ joining $p_0$ and $j_1(p_0)$. Since from  Eq. \ref{eq-antig1} we have $d(p_0,j_1(p_0))>1$, by the Intermediate Value Theorem there exists a point $p^{*}\in C_3^{fin}$ (at least one in each branch) such that $G_{p_0}(p^{*})=d(p^{*},p_0)=1$.
\end{proof}
However, in order to complete the proof of Theorem \ref{prop:fin2} we will show that for any fixed point $p_0\in C_3^{fin}$ there are {\it exactly} two points $p_1,p_2\in C^{fin}_3$ such that 
$$
d(p_0,p_1)=d(p_0,p_2)=1\;\text{and}\;d(p_1,p_2)\neq 1.
$$
To do so, we will mostly use straightforward calculations.

\subsubsection{Proof of Theorem \ref{prop:fin2}}
For a given $(\theta_0,\phi_0)$ such that
$$
\cos(\phi_0+\theta_0/2)=\frac{1+6\cos\theta_0}{8\sqrt{\cos\theta_0}\cos(\theta_0/2)}=f(\theta_0)=f_0=\frac{a_0}{b_0},
$$
we  want to find $(\theta,\phi)$ such that
$$
\cos\left((\phi+\theta/2)-(\phi_0+\theta_0/2)\right)=\frac{1+6\cos\theta\cos\theta_0-2\sin\theta\sin\theta_0}{8\sqrt{\cos\theta\cos\theta_0}\cos\left((\theta+\theta_0)/2\right)}=g(\theta,\theta_0)=g=\frac{A}{B}
$$
and
$$
\cos(\phi+\theta/2)=\frac{1+6\cos\theta}{8\sqrt{\cos\theta}\cos(\theta/2)}=f(\theta)=f=\frac{a}{b}.
$$
In the first place,
$$
g=ff_0+\sin(\phi+\theta/2)\sin(\phi_0+\theta_0/2);
$$
thus we have
$$
(g-ff_0)^2=(1-f^2)(1-f_0^2)
$$
which we write again as
\begin{equation}\label{eq-m-p}
(Abb_0-Baa_0)^2=B^2(b^2-a^2)(b_0^2-a_0^2).
\end{equation}
We calculate
\begin{eqnarray*}
&&
b^2-a^2=-4\cos^2\theta+20\cos\theta-1,\\
&&
b_0^2-a_0^2=-4\cos^2\theta_0+20\cos\theta_0-1,\\
&&
B^2=32\cos\theta\cos\theta_0(1+\cos(\theta+\theta_0)),
\end{eqnarray*}
and
$$
Abb_0-Baa_0=8\sqrt{\cos\theta\cos\theta_0}\left( C\cos(\theta/2)+D\sin(\theta/2)\right),
$$
where
\begin{eqnarray*}
&&
C=\cos(\theta_0/2)\left(c_1\cos\theta+c_2\right),\\
&&
D=\sin(\theta_0/2)\left(d_1\cos\theta+d_2\right),
\end{eqnarray*}
%where
%(Ab-Baf_0)^2=32\cos\theta\left((C^2-D^2)\cos\theta-2CD\sin\theta+C^2+D^2\right),
with
\begin{eqnarray*}
&&
c_1=6(2\cos\theta_0-1),\quad c_2=7-6\cos\theta_0,\\
&&
d_1=10(2\cos\theta_0-1),\quad d_2=-5(2\cos\theta_0+3).
\end{eqnarray*}
Thus Equation \ref{eq-m-p} reads now as
\begin{equation}\label{eq-m-p0}
2(C\cos(\theta/2)+D\sin(\theta/2))^2=(b^2-a^2)(b_0^2-a_0^2)(\cos\theta_0\cos\theta-\sin\theta_0\sin\theta+1).
\end{equation}
%At this point we wish to examine the particular case when $\theta_0=\pm\theta^{*}=\pm\arccos(5/2-\sqrt{6})$. In this case, the right hand side of Equation (\ref{eq-m-p0}) vanishes and thus we simply have
%$$
%C\cos(\theta/2)+D\sin(\theta/2)=0
%$$
%with
%$$
%c_1=12(2-\sqrt{6}),\quad c_2=2(-4+3\sqrt{6}),\quad d_1=10(4-2\sqrt{6}),\quad d_2=-10(4-\sqrt{6}).
%$$
We set $t=\tan(\theta/2)$ and $t_0=\tan(\theta_0/2)$. Then we obtain the 6th degree polynomial equation
\begin{eqnarray}\label{eq-maineq}
P_{t_0}(t)&=&\left(5t_0(7-5t_0^2)\cdot t^3-(31t_0^2-5)\cdot t^2+5t_0(7t_0^2+3)\cdot t-7+5t_0^2\right)^2\\
\notag &&-(-25t_0^4+6t_0^2+15)(-25t^4+6t^2+15)(t_0t-1)^2=0.
%\text{R.H.S}%&=&\frac{b_0^2-a_0^2}{(1+t^2)^3}(-25t^4+6t^2+15)((1-\cos\theta_0)t^2-2\sin\theta_0 t+1+\cos\theta_0)\\
%&=&\frac{2(-25t_0^4+6t_0^2+15)(-25t^4+6t^2+15)(t_0t-1)^2}{(1+t_0^2)^3(1+t^2)^3}
\end{eqnarray}
%and
%\begin{eqnarray*}
%\text{L.H.S}%&=&\frac{2}{(1+t^2)^3}\left(\cos(\theta_0/2)\left((c_2-c_1)t^2+c_1+c_2\right)\right.\\
%&&\left.-\sin(\theta_0/2)\left((d_2-d_1)t^2+d_1+d_2)\cdot t\right)\right)^2\\
%&=&\frac{2(-5t_0(7-5t_0^2)\cdot t^3+(31t_0^2-5)\cdot t^2-5t_0(7t_0^2+3)\cdot t+7-5t_0^2)^2}{(1+t_0^2)^3(1+t^2)^3}.
%\end{eqnarray*}
As a function of variables $t_0$ and $t$, $P$ is a symmetric polynomial: $P_{t_0}(t)=P_t(t_0)$. Moreover, it represents the closed curve which is depicted in the following figure.
\begin{center}
\includegraphics[scale=0.3]{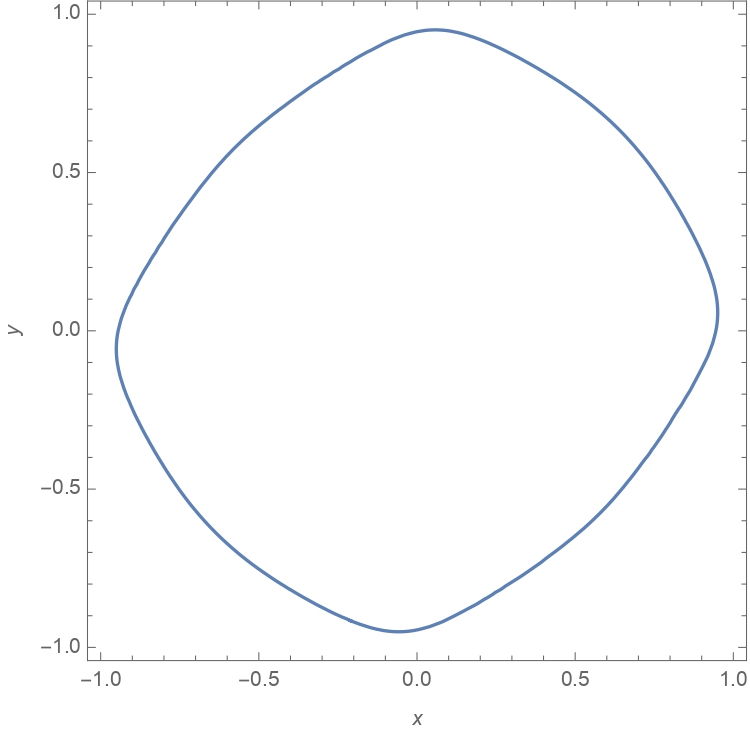}
\begin{figure}[!h]
\centering{The curve $P=0.$ 
}
\end{figure}
\end{center}
The right term of (\ref{eq-maineq}) vanishes when $t_0=\pm t^{*}$, $t^{*}=\tan(\theta^{*}/2)$, $\theta^{*}=\arccos(5/2-\sqrt{6})$. The equation is subsequently reduced to a cubic equation which has only one real root. This root correspond to the upper (resp. lower) point of the above curve. In all other cases we see that the equation has exactly two real roots; one of them is negative and the other is positive. %This can be also verified by explicit calculations: The second derivative of $P$,
%\begin{eqnarray*}
%\frac{1}{2}\cdot\frac{d^2P_{t_0}(t)}{dt^2}&=&375t_0^2(7-5t_0^2)\cdot t^4-100t_0(7-5t_0^2)(31t_0^2-5)\cdot t^3\\
%&&+\left(6(31t_0^2-5)^2+300t_0^2(7t_0^2+3)(7-5t_0^2)\right)\cdot t^2\\
%&&-30t_0\left((31t_0^2-5)(7t_0^2+3)+(7-5t_0^2)^2\right)\cdot t\\
%&&+25t_0^2(7t_0^2+3)^2+2(31t_0^2-5)(7-5t_0^2),
%\end{eqnarray*}
%is always positive in $(-t^{*},t^{*})$ and thus $P_{t_0}$ is convex for all $t_0$. Since 
%$$
%P_{t_0}(-t^{*})>0,\quad P_{t_0}(0)<0,\quad P_{t_0}(t^{*})>0,
%$$
%we have the result by Bolzano's Theorem and convexity of $P_{t_0}$.

The most tractable sub-case here is when $t_0=0$ ($\theta_0=0$). Then (\ref{eq-maineq}) is reduced to
$$
(5t^2-7)^2=15(-25t^4+6t^2+15).
$$  
This gives
$$
t=\pm\sqrt{\frac{1+2\sqrt{3}}{5}}\implies \theta=\pm 2\arctan\left(\sqrt{\frac{1+2\sqrt{3}}{5}}\right)=\pm\vartheta_0.
$$
Now, $\theta_0=0$ corresponds to two points on the curve $C_3^{fin}$, namely the points $p_0=(0,\arccos(7/8))$ and $j_1(p_0)$. From our solution we also have four points, namely
$$
q_0=(\vartheta_0,\phi^+(\vartheta_0)),\quad j_i(q_0),\;i=1,2,3.
$$
We have
\begin{eqnarray*}
&&
d(p_0, q_0)=d(p_0,j_2(q_0))=1,\\
&&
d(j_1(p_0),j_1(q_0))=d(j_1(p_0),j_3(q_0))=1.
\end{eqnarray*}
It is straightforwardly verified that all other distances between these points are different than 1.

In the general case, let with no loss of generality $t_0>0$ corresponding to a fixed point $p_0=(\theta_0,\phi_0)\in C^{fin}_3$. Let also $t_1<0$ and $t_2>0$ the two roots of $P_{t_0}$:
$$
P_{t_0}(t_1)=P_{t_0}(t_2)=0,
$$ 
and these two roots correspond to points $p_1=(\theta_1,\phi_1)$ and $p_2=(\theta_2,\phi_2)$ in $C^{fin}_3$, respectively. Supposing that $d(p_1,p_2)=1$ we would have $P_{t_1}(t_2)=0$. But from symmetry of $P$ we would also have $P_{t_1}(t_0)=0$, hence $P_{t_1}$ has two positive roots, a contradiction. This completes the proof.

\section{Appendix}\label{sec-appendix}
\subsection{Parametric representation of $\C^{fin}_3$}\label{sec-parametric}
The curve $C_3^{fin}$ intersects the $\theta$-axis at points $(\pm\arccos(1/4),0)$ and the $\phi$-axis at points $(0,\;\pm\arccos(7/8))$. To find the range of $\theta$ we put $\phi+\theta/2=0$ and from
$$
\frac{1+6\cos\theta}{8\sqrt{\cos\theta}\cos(\theta/2)}=\pm 1
$$
we find $\theta\in I=[-\theta^*,\theta^*]$ where $\theta^*=\arccos(5/2-\sqrt{6})$.
We set
\begin{equation*}
f(\theta)=\frac{1+6\cos\theta}{8\sqrt{\cos\theta}\cos(\theta/2)},\quad \theta\in I,
\end{equation*}
as in \ref{eq:f}.
Its derivative in ${\rm Int}(I)$ is given by
$$
f'(\theta)=-\sin\theta\cdot\frac{4\cos\theta-1}{32\cos^{3/2}\theta\cos^3(\theta/2)}.
$$
The function $f$:
\begin{itemize}
\item attains its global maximum value 1 at the points $\pm\arccos(5/2-\sqrt{6})$;
\item attains its global minimum value $\sqrt{5/8}$ at the points $\pm\arccos(1/4)$;
\item attains a locally maximum value $7/8$ at 0.
\end{itemize}
Moreover, $f$ is:
\begin{itemize}
\item strictly monotone decreasing at the intervals $$
[-\arccos(5/2-\sqrt{6}),\;-\arccos(1/4)],\quad [0,\;\arccos(1/4)];
$$
\item strictly monotone increasing at the intervals
$$
[-\arccos(1/4),\;0],\quad [\arccos(1/4),\;\arccos(5/2-\sqrt{6})].
$$
\end{itemize}
\begin{center}
\includegraphics[scale=0.3]{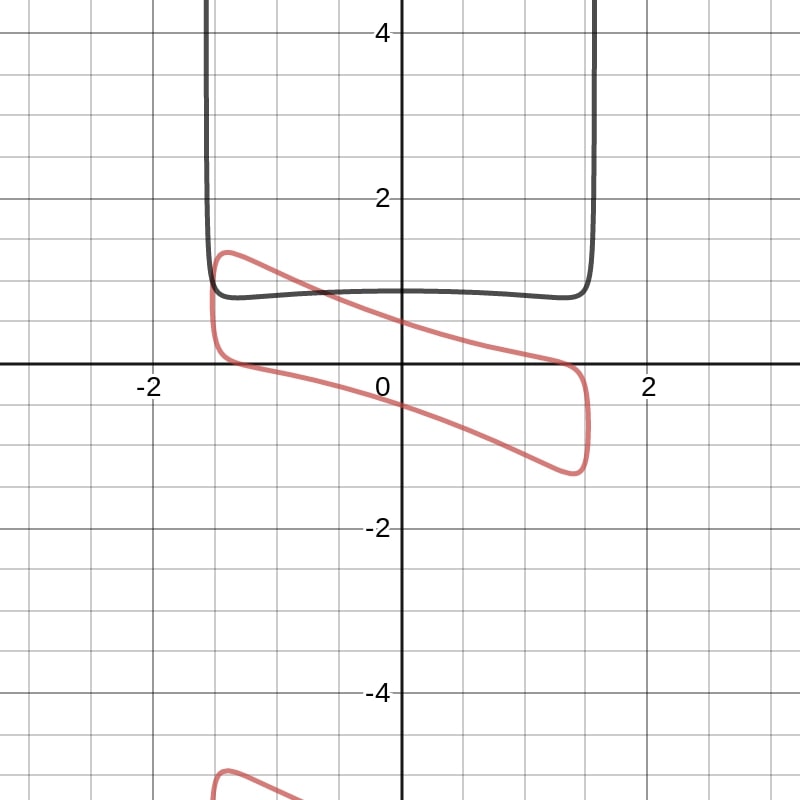}
\begin{figure}[!h]
\centering{The function $f.$ 
}
\end{figure}
\end{center}
We note that
$$
\sqrt{5/8}\le\cos(\phi+\theta/2)\le 1,
$$
which also reads as
$$
-\arccos(\sqrt{5/8})\le \phi+\theta/2\le \arccos(\sqrt{5/8}).
$$
We conclude that  $C_3^{fin}$ is the union of the graphs of the functions $\phi^\pm:I\to\R$ as in \ref{eq:phipm}.

\subsection{Symmetries of $C_3^{fin}$}\label{sec-symetries}
The involution $j:(z,t)\mapsto(\overline{z},-t)$ of the Heisenberg group is restricted to an involution $j_1$ defined in  $C_3^{fin}$ which maps each $p=(\theta,\phi)\in C_3^{fin}$ to the point $j_1(p)=(-\theta,-\phi)\in C_3^{fin}$. That is to say, for each  $\theta\in I,$
\begin{eqnarray*}
&&
(\theta, \phi^+(\theta))\mapsto(-\theta, \phi^-(-\theta))\\ 
&&
(\theta, \phi^-(\theta))\mapsto(-\theta, \phi^+(-\theta)).
\end{eqnarray*}
We call $j_1$ the antipodal involution of $C_3^{fin}$; this defines antipodal pairs of points $(p,j_1(p))$, $p\in C_3^{fin}$.

Let $j_1$ be the antipodal involution and let the function $h:C_3^{fin}\to \R^+$ given by $$h(p)=d^4(p,\;j_1(p)).$$ Then  $h$ as a function of $\theta\in I$ is given by:
$$
h(\theta)=\frac{1}{2}\cdot\frac{8\cos^3\theta-12\cos^2\theta+12\cos\theta+7}{1+\cos\theta},\quad \theta\in I.
$$
The following hold: 
\begin{enumerate} 
\item [{a)}] $h$ attains its global maximum value $15/4$ at the three pairs of antipodal points
\begin{eqnarray*}
&&
q_1^\pm=\pm\left(0,\; \arccos(7/8)\right),\\
&&
 q_2^\pm=\pm\left(\arccos(1/4),\;0\right),\\
 &&
 q_3^\pm=\pm (\arccos(1/4),\;-\arccos(1/4)).
\end{eqnarray*}
\end{enumerate}
Also, 
\begin{enumerate} 
\item[{b)}] $h$ attains its global minimum value $(2\sqrt{6}-3)^2$ at the three pairs of antipodal points
\begin{eqnarray*}
&&
r_1^\pm=\pm(\arccos((5-2\sqrt{6})/2),\;-(1/2)\arccos((5-2\sqrt{6})/2)),\\
&&
r_2^\pm=\pm(\arccos((-1+\sqrt{6})/2),\\
&&-(1/2)\arccos((-1+\sqrt{6})/2)+\arccos(\sqrt{5/8}
\cdot(-2+3\sqrt{6})/5)),\\
&&
r_3^\pm=\pm(\arccos((-1+\sqrt{6})/2),\\
&&-(1/2)\arccos((-1+\sqrt{6})/2)-\arccos(\sqrt{5/8}
\cdot(-2+3\sqrt{6})/5)).
\end{eqnarray*}
\end{enumerate}
Finally,
\begin{enumerate} 
\item [{c)}] (Eq. \ref{eq-antig1}) for all $p\in C^{fin}_3$, 
\begin{equation*}
1.378\approx(2\sqrt{6}-3)^{1/2}\le d(p,j_1(p))\le (15/4)^{1/4}\approx 1.391.
\end{equation*}
\end{enumerate}
\begin{center}
\includegraphics[scale=0.2]{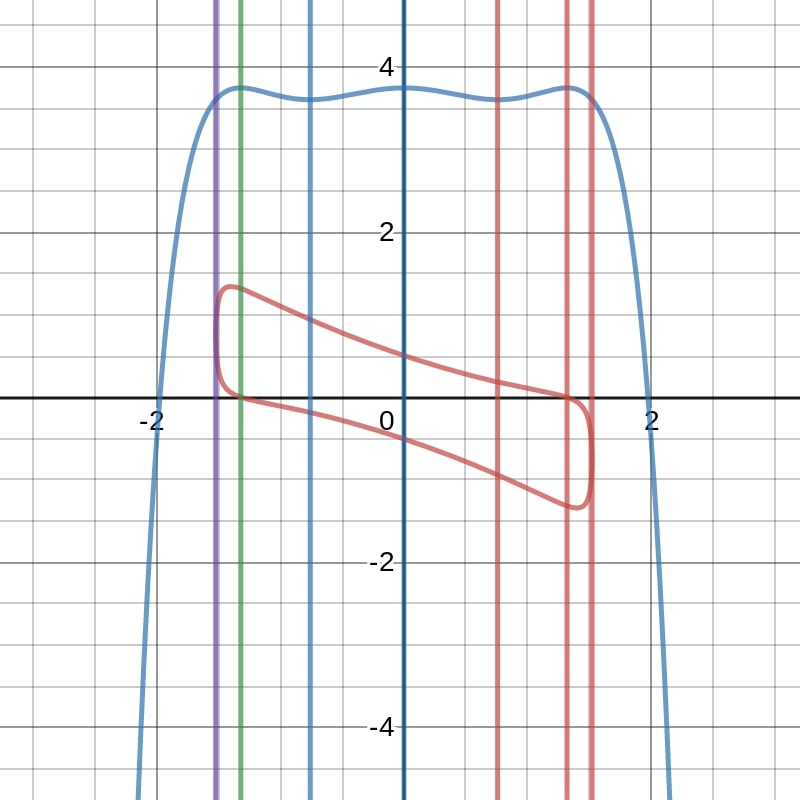}
\begin{figure}[!h]
\centering{The antipodal function $h$.
}
\end{figure}
\end{center}
We conclude from Eq. \ref{eq-antig1} that the Kor\'anyi diametre of $C_3^{fin}$ is greater than 1 and that for any two antipodal points $p,j_1(p)$ we always have $d(p,j_1(p))>1$. It is worth to mention here the following:
\begin{cor}
The curve $C_3^{fin}$ is tangent to the horizontal foliation at points $q_i^\pm$, $i=1,2,3$. Moreover,
\begin{eqnarray*}
&&
d^4(q_1^+,q_2^+)=d^4(q_2^+,q_3^+)=d^4(q_3^+,q_1^-)=d^4(q_1^-,q_2^-)
=d^4(q_2^-,q_3^-)=d^4(q_3^-,q_1^+)=3/8,\\
&&
d^4(q_1^+,q_3^+)=d^4(q_2^+,q_1^-)=d^4(q_3^+,q_2^-)=d^4(q_1^-,q_3^-)
=d^4(q_2^-,q_1^+)=d^4(q_3^-,q_2^+)=9/4,\\
&&
 d^4(q_1^+,q_1^-)=d^4(q_2^+,q_2^-)=d^4(q_3^+,q_3^-)=15/4.
 \end{eqnarray*}
\end{cor}
\begin{proof}
For the first statement, let $\omega=dt+2\Im(\overline{z}dz)$ be the contact form of $\fH$. Then it is written in terms of Kor\'anyi-Reimann coordinates as
$$
\omega=\cos\theta(d\theta+2d\phi).
$$
For $\phi$ as in (\ref{eq:phipm}) we then have
$$
d\theta+2d\phi=\pm\frac{2f'(\theta)}{\sqrt{1-f^2(\theta)}}\;d\theta
$$
and the proof follows from the properties of $f$ and $h$. The proof of the second statement follows after straightforward calculations. 
\end{proof}

The  symmetric involution $j_2$ maps each $p=(\theta,\phi)\in C_3^{fin}$ to the point $j_2(p)=(-\theta,\phi+\theta)$.  It defines symmetric pairs of points $p,j_2(p)$, $p\in C_3^{fin}$. Let
$$
s(p)=d^4(p,j_2(p)), \quad p\in C_3^{fin}.
$$
As a function of $\theta\in I$,
\begin{eqnarray*}
s(\theta)&=&2+4\cos^2\theta+2-8\cos\theta\\
&=&4(1-\cos\theta)^2=16\sin^4(\theta/2).
\end{eqnarray*}
\begin{center}
\includegraphics[scale=0.2]{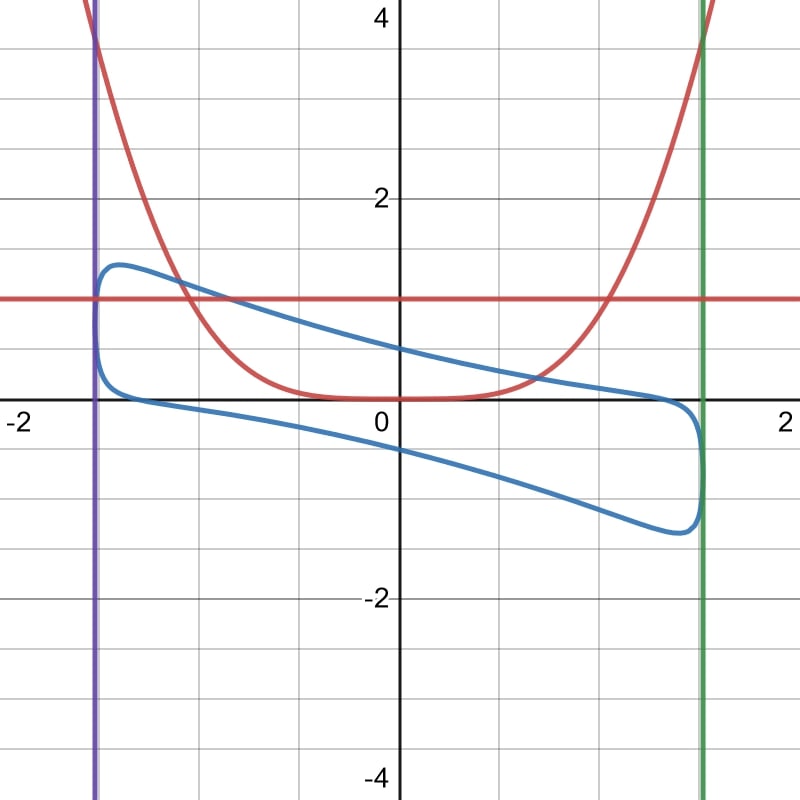}
\begin{figure}[!h]
\centering{The symmetric function $s$.
}
\end{figure}
\end{center}
The function $s$:
\begin{enumerate}
\item attains its global maximum value $4(\sqrt{6}-3/2)^2$ at points $\pm\theta^*$.
\item determines the following pairs of symmetric points with distance 1:
\begin{eqnarray*}
&&
(\pi/3,\;-\pi/6+\arccos(\sqrt{6}/3)),\quad (-\pi/3,\;\pi/6+\arccos(\sqrt{6}/3)),\\
&&
(\pi/3,\;-\pi/6-\arccos(\sqrt{6}/3)),\quad (-\pi/3,\;\pi/6-\arccos(\sqrt{6}/3)).
\end{eqnarray*}
\end{enumerate}
Finally,  the vertical involution $j_3$ of $C_3^{fin}$ maps each $p=(\theta,\phi)\in C_3^{fin}$ to the point $j_3(p)=(\theta,-\phi-\theta)$. Note that  $j_3$ is just the map $(\theta,\phi^+(\theta))\mapsto(\theta,\phi^-(\theta))$; hence vertical points $p$ and $j_3(p)$ lie in the same spherical $\C$-circle ($j_3(p$) is the image of $p$ by some rotation of $\fH$). 
Let $v(p)=d^4(p,j_3(p))$, $p\in C_3^{fin}$.
 As a function of $\theta\in I$,
\begin{eqnarray*}
v(\theta)&=&\frac{\cos\theta}{2(1+\cos\theta)}(-4\cos^2\theta+20\cos\theta-1).
\end{eqnarray*}
\begin{center}
\includegraphics[scale=0.2]{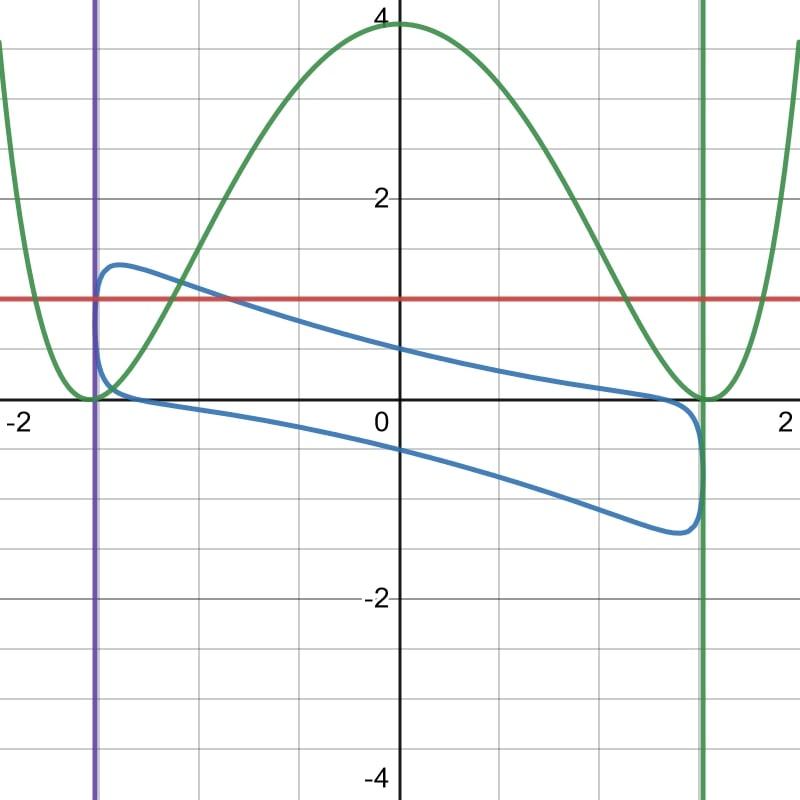}
\begin{figure}[!h]
\centering{The vertical function $v$.
}
\end{figure}
\end{center}
The function $v(\theta)$:
\begin{enumerate}
\item attains its maximum value $15/4$ at $0$;
\item determines the following pairs of vertical points with distance 1:
\begin{eqnarray*}
&&
(\theta_0,\;\phi^+(\theta_0)),\; (\theta_0,\;\phi^-(\theta_0)),\\
&&
(-\theta_0,\;\phi^+(-\theta_0)),\; (-\theta_0,\;\phi^-(-\theta_0)),
\end{eqnarray*}
where $\theta_0=\arccos(0.42...)$ is the absolute value of the two solutions of the equation
$$
4\cos^3\theta-2\cos^2\theta+3\cos\theta+2=0.
$$ 
\end{enumerate}

\end{document}